\documentclass[reqno, 12pt]{amsart} 
\usepackage{array} 
\usepackage{amsmath}
\usepackage{amsfonts}
\usepackage{amssymb}
\usepackage{mathrsfs}
\usepackage{enumerate}
\usepackage{amsthm}
\usepackage{verbatim}
\usepackage{amsmath, amscd}
\usepackage[usenames,dvipsnames]{color}
\usepackage{enumitem}
\usepackage{bm}
\usepackage{xy}
\xyoption{all}
\usepackage{color}
\usepackage{marvosym}
\usepackage{amsbsy}
\usepackage{hyperref}
\usepackage{tikz, tikz-cd}
\usepackage{marginnote}
\usetikzlibrary{matrix,arrows,backgrounds}
\usepackage[backgroundcolor=white]{todonotes}
\usepackage{hyperref}

\newtheorem{theorem}{Theorem}[section]

\newtheorem{lemma}[theorem]{Lemma}
\newtheorem{proposition}[theorem]{Proposition}
\newtheorem{corollary}[theorem]{Corollary}

\newtheorem{conjecture}[theorem]{Conjecture}

\theoremstyle{definition}

\newtheorem{remark}[theorem]{Remark}

\newtheorem{example}[theorem]{Example}

\newcommand{\op}[1]{\operatorname{#1}}

\newcommand{\dbcoh}[1]{\operatorname{D}^{\operatorname{b}}(\operatorname{coh }#1)}

\newcommand{\dabs}{\mathrm{D}}
\newcommand{\gm}{\mathbb{G}_m}

\newcommand{\Hom}{\operatorname{Hom}}

\def\Z{\op{\mathbb{Z}}}
\def\C{\op{\mathbb{C}}}

\def\S{\op{\mathbb{S}}}
\def\Q{\op{\mathbb{Q}}}

\def\O{\op{\mathcal{O}}}
\def\A{\op{\mathbb{A}}}
\def\P{\op{\mathbb{P}}}

\def\Ext{\operatorname{Ext}}

\def\spec{\operatorname{Spec}}

\title[On a Conjecture of Lekili and Ueda]{A maximally-graded invertible cubic threefold that does not admit a full exceptional collection of line bundles}

\author[Favero]{David Favero}
\address{
	\begin{tabular}{l}
		David Favero \\
		\hspace{.1in} University of Alberta, Department of Mathematical and Statistical Sciences \\
		\hspace{.1in} Central Academic Building 632, Edmonton, AB, Canada T6G 2C7 \\
		\hspace{.1in} Korea Institute for Advanced Study \\
		\hspace{.1in} 85 Hoegiro, Dongdaemun-gu, Seoul, Republic of Korea 02455 \\
		\hspace{.1in} Email: {\bf favero@ualberta.ca} \\
	\end{tabular}
}

\author[Kaplan]{Daniel Kaplan}
\address{
  \begin{tabular}{l}
   Daniel Kaplan \\
   \hspace{.1in} University of Birmingham, School of Mathematics \\
   \hspace{.1in} Edgbaston, Birmingham B15 2TT,  United Kingdom \\
   \hspace{.1in} Email: {\bf d.kaplan@bham.ac.uk} \\
  \end{tabular}
}

\author[Kelly]{Tyler L. Kelly}
\address{
  \begin{tabular}{l}
   Tyler L. Kelly \\
   \hspace{.1in} University of Birmingham, School of Mathematics \\
   \hspace{.1in} Edgbaston, Birmingham B15 2TT,  United Kingdom \\
   \hspace{.1in} Email: {\bf t.kelly.1@bham.ac.uk} \\
  \end{tabular}
}

\begin{document}

\begin{abstract}

We show that there exists a cubic threefold defined by an invertible polynomial that, when quotiented by the maximal diagonal symmetry group, has a derived category which does not have a full exceptional collection consisting of line bundles. This provides a counterexample to a conjecture of Lekili and Ueda.
\end{abstract}

\maketitle

\section{Introduction}

Let $\C$ be the complex numbers. We say a polynomial $w \in \C[x_1, \ldots, x_n]$ is invertible if it is of the form 
$$
w = \sum_{i=1}^n \prod_{j=1}^n x_j^{a_{ij}}
$$
where $A = (a_{ij})_{i,j=1}^n$ is a non-negative integer-valued matrix satisfying the following conditions:

\begin{enumerate}
\item the matrix $A$ is invertible over $\Q$;
\item the polynomial $w$ is quasihomogeneous, i.e., there exists \emph{positive} integers $q_j$ such that $d:= \sum_{j=1}^n q_j a_{ij}$ is constant for all $i$; and 
\item the polynomial $w$ is quasi-smooth, i.e., the map $w:\C^n\to\C$ has a unique critical point at the origin.
\end{enumerate}

Let $\gm$ be the multiplicative torus. Consider the following group
\begin{equation}\label{def:gammaw}
\Gamma_w := \{(t_1, \ldots, t_{n+1}) \in \gm^{n+1} \ | \ w(t_1x_1, \ldots, t_nx_n) = t_{n+1}w(x_1, \ldots, x_n) \}.
\end{equation}
This group $\Gamma_w$ acts on $\mathbb{A}^n$ by projecting onto its first $n$ coordinates and then acting diagonally. 
Lekili and Ueda made the following conjecture concerning the bounded derived category associated to the polynomial $w$ and the group $\Gamma_w$.

\begin{conjecture}[Conjecture 1.3 of \cite{LU18}]\label{conjLU}
For any invertible polynomial $w$, the bounded derived category $\dbcoh{X_w}$ of coherent sheaves on the stack
$$
X_w := [(\spec(\C[x_1, \ldots, x_n]/(w))\setminus 0 / \Gamma_w]
$$
has a tilting object, which is a direct sum of line bundles.
\end{conjecture}

In this paper, we show that
\begin{equation}\label{our w}
w = x_1^2x_2 + x_2^2 x_3 + x_3^2x_4+x_4^2x_5 +x_5^2 x_1
\end{equation}
provides a counterexample to this conjecture.  In fact, the maximal length of any exceptional collection of line bundles on $\dbcoh{X_w}$ is 24.  On the other hand, we calculate that 54 line bundles would be required in any full exceptional collection, let alone tilting object. 

\subsection{Relation to current literature and mirror symmetry}

The result above is analogous to the case of toric varieties.  It was asked by King if the derived category of a smooth projective toric variety admits a tilting object which is a direct sum of line bundles.  This later became known as King's conjecture.  The first counterexamples to King's conjecture were provided by Hille-Perling \cite{HP06} then later by Efimov \cite{Efimov} in the Fano case.  Nevertheless, in \cite{Kaw1} Kawamata proved that the derived category of any smooth  projective toric Deligne-Mumford stack has a full exceptional collection.  It just need not consist of line bundles or sheaves for that matter (see  \cite[Remark 7]{Kaw2}).

The Landau-Ginzburg B-model analogue to $\dbcoh{X_w}$ given by the singularity category of $(\C^n, \Gamma_w, w)$ is well-studied in the context of homological mirror symmetry. At present, it is known to have a full exceptional collection \cite{FKK20}. It is also known to have a full strong exceptional collection in certain cases, e.g., when $n\leq 3$ \cite{Kra19} or when $w$ can be written as the Thom-Sebastiani sum of Fermat and loop polynomials \cite{HO18}. This has been desirable in order to establish Homological Mirror Symmetry for mirror pairs of (gauged) Landau-Ginzburg models \cite{Tak09, KST1, KST2, FU09, FU11, LU18, HS19}.

\subsection{Plan of Paper}

In Section~\ref{sec:line bundles}, we show that Picard group of the stack $X_w$ is isomorphic to $\Z \times \Z\!/11 \!\Z$.  In Section~\ref{sec:CRCohom}, we calculate that the Chen-Ruan cohomology of $X_w$ is 54-dimensional.  This implies that the cardinality of any full exceptional collection for $\dbcoh{X_w}$ must be 54 (Corollary~\ref{cor: HH_*}). On the other hand, in Section~\ref{sec:Homs} we find a sharp upper bound of 24 on the cardinality of an exceptional collection for  $\dbcoh{X_w}$ consisting of line bundles.

\subsection{Acknowledgments}

We cordially thank Yank\i{} Lekili, Atsushi Takahashi, Yuki Hirano, and Kazushi Ueda for their communications regarding the topic of this paper. We are especially grateful to Kazushi Ueda for spotting an issue with our initial attempt at finding a counterexample. 

This project was initiated while the three authors visited the Fields Institute during the Thematic Program on the Homological Algebra of Mirror Symmetry. The first-named author was supported by the Natural Sciences and Engineering Research Council of Canada through the Canada Research Chair and Discovery Grant programs.
The second-named author was supported by the Engineering and Physical Sciences Research Council (EPSRC) under Grant EP/S03062X/1. 
The third-named author was also supported by the EPSRC under Grant EP/N004922/2 and EP/S03062X/1.

\section{Line bundles on $X_w$}\label{sec:line bundles}

In order to address Conjecture~\ref{conjLU}, we first  require an explicit description of the Picard group of $X_w$. 

\subsection{The group $\Gamma_w$} 
First, we define the group of diagonal automorphisms of the invertible polynomial $w$ to be
\begin{equation}
G_w :=  \{(t_1, \ldots, t_n) \in \gm^{n} \ | \ w(t_1x_1, \ldots, t_nx_n) = w(x_1, \ldots, x_n) \}.
\end{equation}
This sits in an exact sequence
\begin{equation}
0 \longrightarrow G_w \longrightarrow \Gamma_w \stackrel{\chi_{n+1}}{\longrightarrow} \gm \rightarrow 0
\end{equation}
where $\chi_{n+1}$ is the projection onto the $(n+1)^{\text{th}}$ term of $\Gamma_w$. Indeed, we know that $\chi_{n+1}$ is surjective as, given $\lambda \in \gm$, we have that $(\lambda^{q_1/d}, \ldots, \lambda^{q_n/d}, \lambda) \in \Gamma_w$. 

By Lemma 1.6(B) of \cite{Kra09} for a loop polynomial
$$
w = x_1^{a_1}x_2 + x_2^{a_2} x_3 + \ldots + x_{n-1}^{a_{n-1}} x_n + x_n^{a_n} x_1,
$$
we have $G_w \cong \Z \! / (a_1 \cdots a_n + (-1)^{n+1}) \! \Z$ with generator $(e^{2\pi i \varphi_1}, \ldots , e^{2\pi i \varphi_n})$ where 
\begin{equation}\label{Gw gen}
 \varphi_j := \frac{(-1)^{n+1-j} a_1\cdots a_{j-1}}{a_1 \cdots a_n + (-1)^{n+1}}.
\end{equation}

Recall that $w$ is quasi-homogeneous i.e. we can choose $q_i$ such that $d:= \sum_{j=1}^n q_j a_{ij}$ is constant for all $i$ and such that $\gcd(q_1, \ldots, q_n) = 1$.  This yields a subgroup $J_w \cong \gm$  defined by
\[
f: J_w \rightarrow \Gamma_w; \quad f(\lambda) = (\lambda^{q_{1}}, \ldots, \lambda^{q_n}, \lambda^d)
\]
known as the exponential grading operator in the literature.

Furthermore, the inclusion $f$ gives rise to a split short exact sequence 
\begin{equation}\label{split}
0 \longrightarrow J_w \longrightarrow \Gamma_w \longrightarrow \overline{G_w} \longrightarrow 0
\end{equation}
where $\overline{G_w} := G_w / (J_w \cap G_w)$ is the quotient group. Since $\gcd(q_1,\ldots, q_n) = 1$, there exists $b_i$ with $\sum_{i=1}^n b_iq_i = 1$, which gives rise to the splitting of the exact sequence given by
$$
g:\Gamma_w \to J_w;\quad  g(\lambda_1, \ldots, \lambda_n, \lambda_{n+1}) = \prod_{i=1}^n \lambda_i^{b_i}.
$$
Hence $\Gamma_w \cong J_w \times \overline{G_w}$.

The isomorphism $\Gamma_w \cong J_w \times \overline{G_w}$, gives rise to an intermediate quotient stack associated to $J_w$,
$$
Z_w = [(\spec(\C[x_1, \ldots, x_n]/(w))\setminus 0 / J_w],
$$
which is a hypersurface in the weighted projective stack
$$
 [(\spec(\C[x_1, \ldots, x_n])\setminus 0 / J_w] = \P(q_1: \cdots : q_n).
$$
This allows us to identify $X_w$ with the quotient $[Z_w/\overline{ G_w}]$.

\begin{example} \label{exam: G_w calculation}
Let $w = x_1^2x_2 + x_2^2 x_3 + x_3^2x_4+x_4^2x_5 +x_5^2 x_1$ as in~\eqref{our w}.  Then
$G_w = \Z \! / 33 \! \Z$ with generator
\begin{equation}\label{g}
g = (\zeta, \zeta^{-2}, \zeta^{4}, \zeta^{-8}, \zeta^{16})
\end{equation}
where $\zeta$ is a primitive 33rd root of unity. Here, the intersection $J_w\cap G_w$ is generated by $g^{11} = (\zeta^{11}, \zeta^{11}, \zeta^{11}, \zeta^{11}, \zeta^{11})$. 
Hence $\overline{G_w}$ can be identified with the symmetry group generated by $(\xi, \xi^{9}, \xi^{4}, \xi^{3}, \xi^{5})$ where $\xi$ is a primitive 11th root of unity.
\end{example}

\subsection{The Picard group of $X_w$}

The Grothendieck\textendash Lefschetz theorem allows us to calculate the Picard group of $X_w$ as follows.

\begin{proposition}\label{prop:Picard}
Let $w$ be an invertible polynomial with  $n\geq 5$ and $q_1 =\ldots = q_n = 1$. The Picard group of $X_w$ is  isomorphic to $\Z\times \widehat{\overline{G_w}}$, where $\widehat{\overline{G_w}}$ is the group of characters of $\overline{G_w}$. 
\end{proposition}

\begin{proof}
Since $X_w = [Z_w / \overline{G_w}]$ is a global quotient stack, $\op{Pic}(X_w)$ is nothing more than the $\overline{G_w}$-equivariant Picard group of $Z_w$. Note that there is a (surjective) pullback map 
$$
\op{Pic}(X_w) \stackrel{f}{\rightarrow} \op{Pic}(Z_w)
$$
which just forgets the equivariant structure. By the Grothendieck\textendash Lefschetz Theorem (see e.g. \cite[Corollary 3.2]{Hart70}),   $\op{Pic}(Z_w) \cong \Z$ i.e.\ any line bundle is of the form $\O(n)$.  As $\O(n)$ admits an equivariant structure, the forgetful map $f$ is surjective.

Furthermore, as any two equivariant structures differ by a character of $\overline{G_w}$, we get a short exact sequence
$$
0 \longrightarrow \widehat{\overline{G_w}} \longrightarrow \op{Pic}(X_w) \stackrel{f}{\longrightarrow} \Z \longrightarrow 0.
$$
Since $\Z$ is a projective $\Z$-module, this splits to give the desired isomorphism.
\end{proof}

\begin{example} \label{exam: picard group}
Let $w = x_1^2x_2 + x_2^2 x_3 + x_3^2x_4+x_4^2x_5 +x_5^2 x_1$ so that $\overline{G_w} = \Z \! / 11 \! \Z$. Then by Proposition~\ref{prop:Picard}, we have $\op{Pic}(X_w) \cong \Z\times (\Z \!/11 \! \Z)$.
\end{example}

\section{Dimension of the Hochschild homology of $\dbcoh{X_w}$}\label{sec:CRCohom}
In this section, we compute the dimension of the Chen\textendash Ruan cohomology  of $X_w$ to be 54. This implies that any full exceptional collection for $\dbcoh{X_w}$ must have 54 objects.  

\begin{proposition} \label{prop: HH*}
Let $w = x_1^2x_2 + x_2^2 x_3 + x_3^2x_4+x_4^2x_5 +x_5^2 x_1$. Then $\dim(H^*_{CR}(X_w; \C)) = 54$.  
\end{proposition}

\begin{proof}
As vector spaces, the (ungraded) Chen\textendash Ruan cohomology of $X_w$ is the direct sum of ordinary cohomology groups of twisted sectors
$$
H^*_{CR}(X_w; \C) = \bigoplus_{\gamma \in \Gamma_w} H^*(\{w=0\}_{\gamma} /  \Gamma_w ; \C)
$$
where $\{w=0\}_{\gamma} =\{ x \in  \{w=0\}_{\C^5 \setminus \{0\}} \ | \ \gamma\cdot x = x\}$  \cite[Section 3]{CR11}. 

First, note that, if $\gamma = (\lambda_1, \ldots, \lambda_5)$ so that $\lambda_i \neq 1$ for all $i$, then $\gamma\cdot x \neq x$ for all $x \in \C^5 \setminus \{0\}$. This implies that the twisted sector corresponding to $\gamma$ contributes the cohomology of the empty set, i.e., nothing.

First, we address the twisted sector associated to the identity element $\gamma = e$.
Note that $H^*(\{w=0\} /  \Gamma_w ; \C) = H^*(Z_w ; \C)^{\bar G_w}$, so we must see how $\overline{G_w}$ acts on the cohomology of $Z_w$. Recall that the Hodge diamond of the cubic $Z_{w}$ is of the form
\[
\begin{array}{ccccccc}
 	&  	&   	&  1 	&   	&   	& \\
 	& 	& 0 	&    	& 0 	&	& \\
	& 0 	&	& 1	&	& 0	& \\
    0 	& 	& 5	&	& 5	& 	& 0 \\
	& 0 	&	& 1	&	& 0	& \\    
	& 	& 0 	&    	& 0 	&	& \\
 	&  	&   	&  1 	&   	&   	& \\
\end{array}
\]
This is computed using the Griffiths' residue map  \cite{Gri69}, which also allows us to describe the action of $\overline{G_w}$.  Namely any element $H^{2,1}(Z_w)$ can be written as the residue of a $4$-form
$$
\varphi = \frac{Q}{w} \Omega_0, \qquad \Omega_0 = \sum_{i=1}^5 (-1)^i x_i \op{d}x_1\wedge \ldots \wedge \widehat{\op{d}x_i} \wedge \ldots \wedge \op{d}x_5.
$$
where $Q$ is a degree 1 polynomial in $\C[x_1,\ldots, x_5]$. By looking at the action by the generator $\rho$ of $\overline{G_w}$, we can see that $w$ and $\Omega_0$ are invariant under its action; however, no degree $1$ polynomial is, so all of $H^{2,1}(Z_w;\C)$ is not $\overline{G_w}$-invariant. Analogously, the cohomology $H^{1,2}(Z_w;\C)$ is not $\overline{G_w}$-invariant. The hyperplane classes, on the other hand, are all invariant cycles, so 
$$
\dim H^*(\{w=0\} / \Gamma_w ; \C) = 4.
$$

Lastly, there are 50 non-identity elements 
$$
S := \{ (\rho\tau^{-1})^a, (\rho\tau^{-9})^a, (\rho\tau^{-4})^a, (\rho\tau^{-3})^a, (\rho\tau^{-5})^a \ | \ 1 \leq a \leq 10\} \subseteq \Gamma_w
$$
with a fixed point where $\rho:= (\xi, \xi^{9}, \xi^{4}, \xi^{3}, \xi^{5})$ is the generator of $\overline{G_w}$ and $\tau = (\xi, \xi, \xi, \xi, \xi)$.
In fact, each has a single fixed point and hence contributes 1 dimension to the Chen-Ruan cohomology. 

We conclude that $\dim(H^*_{CR}(X_w; \C))= 4 + |S| = 4+50 = 54$.
\end{proof}

This proposition implies the following corollary.

\begin{corollary} \label{cor: HH_*}
For $w$ as defined in~\eqref{our w}, we have that $\dim(\text{HH}_*(\dbcoh{X_w})) = 54$. In particular, any full exceptional collection for $\dbcoh{X_w}$ has precisely 54 objects.
\end{corollary}
\begin{proof}
By an unpublished result of To\"en (reproven in \cite[Proposition 3.16]{HLP16}),
\[
\dim(\text{HH}_*(\dbcoh{X_w})) = \dim(H^*_{CR}(X_w; \C)) = 54.
\]
The fact that any full exceptional collection must have 54 objects follows from additivity of Hochschild homology under semi-orthogonal decomposition.
\end{proof}

\begin{remark}
In \cite[Theorem 1.1]{FKK20}, the authors prove that there is a strong exceptional collection for the singularity category $\dabs[\A^5, \Gamma_w, w]$. It is of length 32, the Milnor number of its mirror LG-model. By the equivariant version of Orlov's theorem (proven by Hirano \cite[Theorem 1.3]{Hir17}), it follows that $\dbcoh{X_w}$ has a full exceptional collection of length $32 + 2(11) = 54$.  From this it also follows that any full exceptional collection must have 54 objects.
\end{remark}

\section{Computations of $\Ext$ between line bundles on $X_w$}\label{sec:Homs}
By Corollary~\ref{cor: HH_*}, any full exceptional collection for $\dbcoh{X_w}$ has 54 objects. However, in this section we show that an exceptional collection consisting of line bundles on $X_w$ has at most 24 objects (and remark that this bound is achieved).

\begin{lemma} \label{lem: homs between line bundles}
For $a \geq 0$, $\Hom(\O, \O(a, b)) \neq 0$ unless $a= 0$ and $b \neq 0$ or 
\[
(a, b) \in  \mathbb{X} : = \{ (1, 0), (1, 2), (1, 6), (1, 7), (1, 8), (1, 10), (2, 0) \}.
\] 
\end{lemma}

\begin{proof}
Observe that $\Hom(\O, \O(a, b))$ is the space of bidegree $(a, b) \in \Z \times \Z \! /11 \! \Z$ polynomials in $\C[x_1,x_2,x_3,x_4,x_5]/(w)$. By Example \ref{exam: G_w calculation}, $\overline G_{w} = \langle  (\xi, \xi^{9}, \xi^{4}, \xi^{3}, \xi^{5})  \rangle \cong \Z \! / 11 \!\Z$ where $\xi$ is a primitive 11th root of unity. Hence,
\[
 \deg(x_1) = (1, 1), \ \  \deg(x_2) = (1, 9), \ \ \deg(x_3) = (1, 4), \ \ \deg(x_4) = (1, 3), \ \ \deg(x_5) = (1, 5).
 \]
 So Table \ref{table: bigraded monomials} exhibits an element in $\Hom(\O, \O(a, b))$ for $1 \leq a \leq 3$, unless $(a, b) \in \mathbb{X}$. We conclude that $\Hom(\O, \O(a, b))$ is non-zero for $a \geq 3$ by multiplying any monomial in $\Hom(\O, \O(3, b-a+3))$ by $x_{1}^{a-3}$. 
\end{proof}

\begin{table}[h]
\caption{The $(a, b)$th entry is an $(a, b)$-bigraded monomial in $\C[x_1,x_2,x_3,x_4,x_5]/(w)$}
\begin{tabular}{c | c | c |  c |  c | c |  c |  c |  c |  c |  c | c }
&  \multicolumn{11}{c}{$\Z \!/ 11 \!\Z$-grading} \\
$\Z$-grading &         0 	&    1 	&    2   	 	&    3    	&    4    		&    5   	 	&    6   	 	&    7    		&    8   	 	&    9    	&    10 \\
\hline \hline
0 & 1 &&&&&&&&&& \\
1 &			&  $x_1$ 	&			& $x_4$ 	& $x_3$ 		& $x_5$ 		& 			&			&			& $x_2$	&	\\
2 &                    &   $x_2x_4$ &$ x_1^2$		& $x_5x_2$ 	&  $x_1 x_4$ 	&  $x_1 x_3$ 	& $x_1 x_5$ 	& $x_3 x_4$ 	& $x_3^2$		&  $x_3 x_5$ & $x_1 x_2$ \\
3 &  $x_1^2 x_2$ & $x_3^3$  	&  $x_1 x_2 x_4$	&   $x_1^3 $	&$x_1 x_2 x_5$&$ x_1^2 x_4$	&  $x_1^2 x_3$	&  $ x_1^2 x_5$ & $x_1 x_3 x_4 $& $x_1x_3^2$ & $x_1 x_3 x_5$\\
\end{tabular}
 \label{table: bigraded monomials}
\end{table}

\begin{lemma} \label{lem: Serre duality}
For $a\geq2$, we have that $\Ext^{3}(\O(a, b), \O) \neq 0$ unless $a=2$ and $b \neq 0$ or 
$$
(a, b)  \in \mathbb{X}' := \{ (3, 0), (3, 2), (3, 6), (3, 7), (3, 8), (3, 10), (4,0) \}.
$$
\end{lemma}

\begin{proof}
By adjunction, the canonical bundle is $\O(-2, 0)$.  
Therefore by Serre duality, 
\begin{align*}
\Ext^{i}(\O(a, b), \O) 	& \overset{\text{Serre}}{\cong} \Ext^{3-i}( \O, \O(a, b) \otimes_{\O} \O(-2, 0))^{*} \\
					& \overset{\textcolor{white}{Serre}}{\cong} \Ext^{3-i}( \O, \O(a-2, b)))^{*}. 
\end{align*}
The result follows from Lemma~\ref{lem: homs between line bundles}.
\end{proof}

\begin{proposition}\label{bound on exceptional line bundles}
An exceptional collection of line bundles in $\dbcoh{X_w}$ has at most 24 objects, and hence cannot be full (by Corollary~\ref{cor: HH_*}). 
\end{proposition}

\begin{proof}
By Example \ref{exam: picard group}, any line bundle on $X_{w}$ is of the form $\O(a, b)$ for $(a, b) \in \Z \times \Z \!/11 \!\Z$.
Let $\mathcal{E}$ denote an exceptional collection of line bundles and take the minimal $a$ such that $\O(a, b) \in \mathcal{E}$ for some $b \in \Z \!/ 11 \!\Z$. Since $\mathcal{E} \otimes \O(-a, -b)$ is an exceptional collection, we can assume $(a, b) = (0, 0)$. 

Notice, $\mathcal{E}$ cannot have an object of the form $\O(a, b)$ for $a \geq 5$ as, by Lemma \ref{lem: homs between line bundles}, $\O(a,b)$ receives a non-zero map from $\O$ and, by Lemma \ref{lem: Serre duality}, there is a non-trivial 3-extension of $\O$ by $\O(a, b)$. 

By Table \ref{table: bigraded monomials}, observe that if $b \neq b'$ then for any $a$, one has non-zero elements 
\[
f_1 \in \Hom(\O(a, b), \O(a+2, b')) \hspace{1cm} \text{and} \hspace{1cm} f_2 \in \Hom(\O(a, b'), \O(a+2, b)).
\]
Therefore, denoting by $\S(a, b)$ the Serre functor applied to the identity map on $\O(a, b)$, one has a loop:
\[
\O(a, b) \overset{f_1}{\longrightarrow} \O(a+2, b') \overset{\S(a+2, b')}{\longrightarrow} \O(a, b') \overset{f_2}{\longrightarrow} \O(a+2, b) \overset{\S(a+2, b)}{\longrightarrow} \O(a, b), 
\] 
We conclude that $\mathcal{E}$ cannot have a quadruple of objects 
\[
\{ \O(a, b), \O(a, b'), \O(a+2, b), \O(a+2, b') \}.
\]
 For example, taking $a=0$ (respectively $a=1$) $\mathcal{E}$ cannot have multiple objects with $a=0$ and $a=2$ (respectively $a=1$ and $a=3$).
 This forces there to be at most $12$ lines bundles in $\mathcal E$ with $a=0,2$ and $a=1,3$ respectively.

Now, again by Lemma~\ref{lem: Serre duality}, $\mathcal{E}$ cannot have an object of the form $\O(a, b)$ for $a \geq 4$ except $(a, b) = (4, 0)$.
Hence, we can have at most $1$ more object.  But if $\O(4,0) \in \mathcal{E}$, Lemma~\ref{lem: Serre duality} also forces $\O(0,b) \notin \mathcal E$ for $b\neq 0$.
Hence, if we already have $12$ lines bundles in $\mathcal E$ with $a=0,2$ then $\O(2,b) \in \mathcal E$ for all $b$.  This gives a contradiction as $\O, \O(2,0), \O(4,0)$ also form a loop 
\[
\O  \overset{x_1^2x_3x_5}{\longrightarrow} \O(4, 0)  \overset{\S(4,0)}{\longrightarrow} \O(2, 0)  \overset{\S(2,0)}{\longrightarrow} \O
\]
and therefore cannot be in the same exceptional collection.
We conclude that this 1 additional object cannot take us beyond $24$ exceptional objects.
 \end{proof}

\begin{remark}
The upper bound of 24 exceptional objects is sharp.  It is achieved by the exceptional collection drawn below.  This exceptional collection is not strong, however, we only draw the degree 0 maps for aesthetic simplicity.  The required vanishing can be checked using Lemmas~\ref{lem: homs between line bundles} and~\ref{lem: Serre duality}, and the fact that $\Ext^1, \Ext^2$ vanish for line bundles on a 3-fold hypersurface in projective space (e.g. using the long exact sequence for the divisor).
\[
\resizebox{.95\textwidth}{!}{
\xymatrix{
\O(1, 2)[-3] \ar[rdd] & 							&							& \O(1, 1) 	\ar[rrr] \ar[]!<-3ex,1ex>;[rrrd]!<-3ex,1ex>   \ar[]!<2ex,1ex>;[rrrdddd]!<-3.5ex,1ex>			&&& \O(2, 2) \ar[rdd]		 &				&							& \O(2, 1)[3] \\
\O(1, 6)[-3] \ar[rd]	&							& 							& \O(1, 3)	\ar[rrr] \ar[]!<1ex,1ex>;[rrrd]!<-3ex,1ex>   \ar[]!<2ex,1ex>;[rrrdd]!<-3ex,1ex> 		 		&&& \O(2, 6) \ar[rd]	 	& 				&							& \O(2, 3)[3] \\
\O(1, 7)[-3] \ar[r] & \O(2, 0)[-3] \ar[r] & \O(0,0) \ar[ruu] \ar[ru] \ar[r] \ar[rd] \ar[rdd]				& \O(1, 4) \ar[rrr] \ar[]!<2ex,-1ex>;[rrruu]!<-3ex,0.5ex>  \ar[]!<1ex,0ex>;[rrrd]!<-3ex,0.5ex>			&&& \O(2, 7) \ar[r] & \O(3, 0) \ar[r] & \O(1, 0)[3] \ar[ruu] \ar[ru] \ar[r] \ar[rd] \ar[rdd] 		& \O(2, 4)[3] \\
\O(1, 8)[-3] \ar[ru]	&							& 							& \O(1, 5)	\ar[rrr] \ar[]!<2ex,-1ex>;[rrruu]!<-3ex,0.5ex>  \ar[]!<1ex,0.5ex>;[rrrd]!<-3ex,0.5ex>				&&& \O(2, 8) \ar[ru]		 & 				&							& \O(2, 5)[3] \\
\O(1, 10)[-3] \ar[ruu] & 							& 							& \O(1, 9)	\ar[rrr] \ar[]!<2ex,-1ex>;[rrruuuu]!<-2.5ex,0.5ex> \ar[]!<3ex,0ex>;[rrruu]!<-2.5ex,0.5ex>		&&& \O(2, 10) \ar[ruu]		 & 				&							& \O(2, 9)[3]
}
}
\]
\end{remark}

\end{document}